\documentclass[11pt, letterpaper]{article}

\usepackage[margin = 1in]{geometry}
\usepackage{amsthm,amssymb,amsmath,graphicx,mathtools,enumerate}
\usepackage{caption}
\usepackage{subcaption}
\usepackage[shortlabels]{enumitem} 
\usepackage[british]{babel}
\usepackage[utf8x]{inputenc}
\usepackage{enumitem}

\usepackage[usenames,dvipsnames]{color}

\theoremstyle{plain}
\newtheorem{theorem}{Theorem}

\newtheorem{lemma}[theorem]{Lemma}
\newtheorem{conjecture}[theorem]{Conjecture}
\newtheorem{claim}[theorem]{Claim}
\newtheorem{observation}[theorem]{Observation}

\theoremstyle{definition}
\newtheorem{definition}[theorem]{Definition}
\newtheorem{setup}[theorem]{Setup}

\newtheorem*{procedure*}{Random colouring procedure}

\numberwithin{equation}{section}

\newcommand{\sm}{\setminus}
\renewcommand{\subset}{\subseteq}

\newcommand{\NATS}{\mathbb{N}}

\newcommand{\K}{\mathrm{Keep}}
\newcommand{\R}{\mathrm{Retain}}

\newcommand{\Eq}{\mathrm{Eq}}

\renewcommand{\Pr}{\mathbf{Pr}}
\newcommand{\Exp}{\mathbf{E}}

 \newcommand{\Vq}{\mathrm{Vq}}

 \renewcommand{\l}{\ell}
 \newcommand{\p}{p}
 \newcommand{\cL}{\mathcal{L}}
 \newcommand{\Res}{\mathrm{Reserve}}

\title{An Improved Bound for the Linear Arboricity Conjecture}

\author{Richard Lang and Luke Postle}
\begin{document}

\maketitle

\begin{abstract}
	In 1980, Akiyama, Exoo and Harary posited the Linear Arboricity Conjecture which states that any graph $G$ of maximum degree $\Delta$ can be decomposed into at most $\left\lceil \frac{\Delta}{2}\right\rceil$ linear forests. 
	(A forest is linear if all of its components are paths.)
	In 1988, Alon proved the conjecture holds asymptotically. 	The current best bound is due to Ferber, Fox and Jain from 2020 who showed that $\frac{\Delta}{2}+ O(\Delta^{.661})$ suffices for large enough $\Delta$. Here, we show that $G$ admits a decomposition into at most $\frac{\Delta}{2}+ 3\sqrt{\Delta} \log^4 \Delta$ linear forests provided $\Delta$ is large enough.
	
	Moreover, our result also holds in the more general list setting, where edges have (possibly different) sets of permissible linear forests. Thus our bound also holds for the List Linear Arboricity Conjecture which was only recently shown to hold asymptotically by Kim and the second author. Indeed, our proof method ties together the Linear Arboricity Conjecture and the well-known List Colouring Conjecture; consequently, our error term for the Linear Arboricity Conjecture matches the best known error-term for the List Colouring Conjecture due to Molloy and Reed from 2000. This follows as we make two copies of every colour and then seek a proper edge colouring where we avoid bicoloured cycles between a colour and its copy; we achieve this via a clever modification of the nibble method.
	
\end{abstract}

\section{Introduction}
A \emph{linear forest} is a disjoint union of paths.
We study the decomposition of graphs into few linear forests.
For instance, by Vizing's theorem any graph $G$ of maximum degree $\Delta$ can be decomposed into at most $\Delta+1$ matchings and hence $\Delta+1$ linear forests.
In 1980, Akiyama, Exoo and Harary~\cite{AEH80} conjectured that one can do much better:
\begin{conjecture}[Linear Arboricity Conjecture]\label{con:linear-arboricity}
	Any graph $G$ of maximum degree $\Delta$ admits a decomposition into at most $\lceil \Delta/2\rceil$ linear forests.
\end{conjecture}
The Linear Arboricity Conjecture has received a fair amount of attention throughout the years and has been shown for many particular families of graphs (see~\cite{FFJ20}). Here we are interested in general bounds for large maximum degree $\Delta$.

The first result of this type was obtained by Alon~\cite{Alo88} in 1988 and states that any graph $G$ of maximum degree at most $\Delta$ can be decomposed into at most $\frac{\Delta}{2} + O\left(\Delta \cdot \frac{\log \log \Delta}{\log \Delta}\right)$ linear forests. The error term was improved to $O(\Delta^{2/3} \log^{1/3} \Delta)$ by Alon and Spencer~\cite{AS16} in 1992. In 2020, Ferber, Fox and Jain~\cite{FFJ20} reduced this further to $O(\Delta^{.661})$. The first main result of this paper is the following improvement.

\begin{theorem}\label{thm:Ordinary}
Any graph $G$ of maximum degree $\Delta$ admits a decomposition into at most $\frac{\Delta}{2} + 3 \Delta^{1/2} \log^4\Delta $ linear forests.
\end{theorem} 

In fact, we will prove Theorem~\ref{thm:Ordinary} in the more general setting of list colouring, where edges have (possibly different) sets of permissible linear forests. Indeed, the list version of Conjecture~\ref{con:linear-arboricity}, called the List Linear Arboricity Conjecture, was formulated by An and Wu in 1999 (see~\cite{AW09}). In 2021, Kim and the second author~\cite{KP17} proved that the List Linear Arboricity Conjecture holds asymptotically, that is if every edge $e$ of $G$ has a list $L(e)$ of size at least $\frac{\Delta}{2} +o(\Delta)$, then there exists a colouring from the lists such that each colour class is a linear forest. Our main result also improves this error term to $O(\Delta^{1/2} \log^4\Delta)$.

Indeed, here we tie together the Linear Arboricity Conjecture and the well-known List Colouring Conjecture posited independently by many researchers in the 1970s and 1980s.

\begin{conjecture}[List Colouring Conjecture]\label{conj:LCC}
If $G$ is a graph of maximum degree $\Delta$ and $L$ is a list assignment to $E(G)$ such that $|L(e)|\ge \Delta+1$ for every edge $e$ of $G$, then there exists an $L$-colouring of the edges of $G$.
\end{conjecture}

In 1996, Kahn~\cite{Kah96} proved that the List Colouring Conjecture holds asymptotically using R\"odl's famous nibble method. In 2000, Molloy and Reed~\cite{MR00} improved the error term to $O(\Delta^{1/2} \log^4\Delta)$ which is still the best known bound to date. Kim and the second author~\cite{KP17} showed how the List Colouring Conjecture and List Linear Arboricity Conjecture are related via the following definition. 

Consider a graph $G$ with \emph{lists of colours} $L(e) \subset \NATS$ assigned to its edges.
A \emph{$t$-edge-colouring} of $G$ from $L(e)$ is an assignment of colours to the edges of $G$ such that each edge receives a colour from its list and no vertex is incident to more than $t$ edges of the same colour.

Note that, if all lists are the same, then a decomposition of $G$ into linear forests is equivalent to a $2$-edge-colouring that does not induce any monochromatic cycles. Kim and the second author~\cite{KP17} noted that results for list $t$-edge-colouring follow from results for the List Colouring Conjecture by replacing every colour with $t$ copies of itself.

Here then is the main result of this paper which shows we can also avoid monochromatic cycles in a $2$-edge-colouring (equivalently bicoloured cycles between a colour and its copy).

\begin{theorem}\label{thm:main}
	Let $G$ be a graph with sufficiently large maximum degree $\Delta$. If $L$ is a list assignment to the edges of $G$ such that $$|L(e)| \geq \frac{\Delta}{2}+ 3\sqrt{\Delta} \log^4 \Delta$$
	for every edge $e$ of $G$, then there is a $2$-edge-colouring of $G$ from $L(e)$ without any monochromatic cycles.
\end{theorem}

The proof employs the nibble method and we purposely wrote it to follow the proof of Molloy and Reed. The key new idea is to modify the random colouring procedure to also remove colours from lists if they would create a bicoloured cycle between a colour and its copy. We will show that this removal has a negligible effect on the list sizes during the nibble process (indeed it has an even smaller effect than the concentration error terms from the main part of the procedure). 

However, we were not able to remove such colours when the cycles are too long since this would create too large of a dependency when applying the Lov\'asz Local Lemma. To handle this, we remove colours that would form short bicoloured cycles (say with at most $O(\log \Delta)$ newly coloured edges) or bicoloured paths (with roughly $\log \Delta$ newly coloured edges). We argue that there are not too many of these paths since we introduce an activation probability (say $1/2$) to colouring edges, and hence these paths will also yield a negligible loss in list size.

The rest of this paper is organized as follows. In the next section, we discuss our approach in more detail and derive Theorem~\ref{thm:main} from Claim~\ref{cla:nibbler}, which is the centerpiece of the proof. In Section~\ref{sec:nibbler}, we prove Claim~\ref{cla:nibbler}.
 
\section{Proof of the main result}\label{sec:nibble}
In this section, we prove Theorem~\ref{thm:main}.
Let $G$ be a graph of sufficiently large maximum degree $\Delta$ and consider an assignment of lists $\cL'$ with 
$|\cL'(e)| \geq \frac{\Delta}{2}+ 3\sqrt{\Delta} \log^4 \Delta$ for every edge $e$ of $G$.
Our goal is to find a $2$-edge-colouring of $G$ from $\cL$, which does not contain monochromatic cycles.
It is instructive to approach this problem via $1$-edge-colourings.
Consider the  list product $\cL(e) = \cL'(e) \times \{1,2\}$, which satisfies
\begin{align}\label{equ:size-cL}
	|\cL(e)| \geq \Delta+ 6\sqrt{\Delta} \log^4 \Delta
\end{align}
for every edge $e$.
We say that colours $(c,1)$ and $(c,2)$ are \emph{twins}.
Molloy and Reed~\cite{MR00} refined an approach of Kahn~\cite{Kah96} to obtain a $1$-edge-colouring of $G$ from $\cL$ by iteratively colouring the edges using a random colouring procedure.
Note that this translates to a $2$-edge-colouring from $\cL'$.
However such a colouring could potentially contain a monochromatic cycle.
Fortunately, we can prevent such cycles from arising by a minor but clever modification of the random colouring process.
In the following we explain this idea in more detail.

The rest of this section is organized as follows.
In the next subsection, we discuss how the random colouring approach can be modified to account for monochromatic cycles.
In Subsection~\ref{sec:adapting-former-work}, we embed our modification into a standard exposition of the iterative colouring procedure of Molloy and Reed~\cite{MR13} and then finish the proof of Theorem~\ref{thm:main}.

\subsection{Modifying the random colouring procedure}
The iterative random colouring procedure, which has been used to obtain $1$-edge-colourings, is a fairly straight forward process.
In each step, a random colour from the lists of (still available) colours is assigned to each uncoloured edge.
Then all arising conflicts are uncoloured.
It can be shown that with positive probability, this results in a  well-behaved partial $1$-edge-colouring allowing for the repetition of the process until most edges of the graph have been coloured.
To prove Theorem~\ref{thm:main}, we need in addition to prevent monochromatic cycles from emerging.

Hence we modify the procedure as follows.
In each step, we also uncolour edges that are on cycles whose colour assignments have left them `dangerously close' to becoming monochromatic under $\cL'$.
As we will see, the error terms resulting from this change are dwarfed by already existing error terms.
Hence the procedure can be carried out essentially in the same way as before.
In what follows, we introduce the appropriate notation to capture this idea and discuss further details.

Let us remark that, for the remainder of the proof, we focus our analysis largely on the list product $\cL(e)$, whose size is bounded in inequality~\eqref{equ:size-cL}.
We say that an edge-colouring is \emph{partial} for $G$, if it is an edge colouring of a subgraph of $G$.
Throughout the colouring procedure, we have to keep track of the overall status of the colouring.
This is why the following definitions are stated in the context of lists $L(e)\subset \cL(e)$ of (yet unused) colours and a partial edge-colouring $\gamma$ of $G$.

\begin{definition}[Dangerous paths]
	Let $\gamma$ be a partial $1$-edge-colouring of $G$.
	Consider vertices $u,v$ and twin colours $c,c'$.
	We say that a path $P$ is \emph{dangerous} for $(u,v,c)$ under $\gamma$, if 
	\begin{itemize}
		\item $P$ starts with $u$ and ends with $v$  and
		\item the edges of $P$ are coloured alternately with $c$ and $c'$, starting and ending with $c'$.
	\end{itemize}
\end{definition}
To motivate this definition, consider a dangerous path $P$ for some $(u,v,c)$ such that $e=uv$ is an uncoloured edge with $c \in L(uv)$.
Observe that if $e$ is coloured with $c$, then $P+e$ would be a monochromatic cycle with respect to the original lists $\cL'$.
Such peril is not acceptable.
We will therefore monitor certain suspicious paths, which are, informally speaking, candidates for becoming dangerous.
During our colouring procedure, we remove $c$ from $L(e)$, whenever a dangerous path for $(u,v,c)$ emerges.

\begin{definition}[Suspicious paths]\label{def:suspicious}
	Let $\gamma$ be a partial $1$-edge-colouring of $G$, consider sublists $L(e)\subset \cL(e)$ of the product lists and let $c,c'$ be twin colours.
 	A path $P=(v_1,\dots,v_s)$ is \emph{suspicious} for $(u,v,c)$ if the following holds:
	\begin{itemize}
	\item $P$ starts with $u$, ends with $v$ and its last edge $v_{s-1}v_s$ is uncoloured.
 	\item If $v_iv_{i+1}$ is uncoloured, then 
		\begin{align*}
		L(v_iv_{i+1}) \ni
		\begin{cases}
		c' &\text{if $i$ is odd and}\\
		c &\text{if $i$ is even}.
		\end{cases}
	\end{align*}
 	\item If $v_iv_{i+1}$ is coloured, then 
	 	\begin{align*}
	 	\gamma(v_iv_{i+1})=	
	 	\begin{cases}
	 	c' &\text{if $i$ is odd and}\\
	 	c &\text{if $i$ is even}.
	 	\end{cases}
	 	\end{align*}
	\end{itemize}
	The \emph{uncoloured length} of $P$ is the number of its uncoloured edges.
	We let $P_{L,\gamma}(u,v,c;k)$ be the union of suspicious for $(u,v,c)$ paths of uncoloured length $k$ and set $P_{L,\gamma}(u,c;k) := \bigcup_{v \in G} P_{L,\gamma}(u,v,c;k)$.
\end{definition}

Note that preventing dangerous paths introduces new errors in the parameters that guide the original iterative colouring procedure.
Fortunately however, the overall effects of this alteration turns out to be negligible in comparison to other error terms.
This is because, firstly, only suspicious paths can become dangerous and, secondly, the number of suspicious paths is dominated by the reciprocal of the probability of a suspicious path becoming dangerous.
Let us now formalize this discussion.

As usual in this line of research, we can restrict our analysis to the neighbours that are relevant to specific colours.
\begin{definition}(Colour neighbours)
	Given  a partial $1$-edge-colouring $\gamma$ of $G$ from some lists $L(e)\subset \cL(e)$, we define the \emph{colour neighbours} of a vertex $v$ and a colour $c$, denoted by $N_{G,L,\gamma}(v,c)$, to be the set of uncoloured edges $f$ incident to $v$ with $c \in L(f)$.
\end{definition}
We begin by bounding the number of suspicious paths.
\begin{observation}\label{obs:number-suspicious}
	Let $\gamma$ be a partial $1$-edge-colouring of $G$ and consider sublists $L(e)\subset \cL(e)$ of the product lists.
	Suppose that there is a constant $N$ such that $|N_{G,L,\gamma}(w,c)| \leq N$ for every vertex $w$ and colour $c$.
	Then we have $|P_{L,\gamma}(u,v,c;k)|\leq N^{k-1}$ and $|P_{L,\gamma}(u,c;k)|\leq N^{k}$ for all vertices $u,v$ colours $c$ and integers $k$.
\end{observation}
\begin{proof} 
	We show the observation for $P_{L,\gamma}(u,c;k)$ by induction on the uncoloured length $k$.
	The base case $k=1$ is simple.
	Since $\gamma$ is a partial $1$-edge-colouring of $G$, there is a unique edge maximal path $P$ starting from $u$, which is coloured with $c$ and $c'$ starting with $c'$.
	Suppose first that the last edge of $P$ is coloured with $c$.
	By assumption there are at most $N$ possibilities to continue $P$ with an edge whose list contains $c'$.
	Of course, the same conclusion holds when the last edge of $P$ is coloured with $c'$.
	We have thus verified the induction start $|P_{L,\gamma}(u,c;1)|\leq N$.
	For the induction step, observe that, for $k\geq 2$, every path of $P_{L,\gamma}(u,c;k)$ is the extension of a path of $P_{L,\gamma}(u,c;k-1)$.
	Using the same argument as before, we see that every path of $P_{L,\gamma}(u,c;k-1)$ extends to at most $N$ paths from $P_{L,\gamma}(u,c;k)$.
	Hence, by induction hypothesis $|P_{L,\gamma}(u,c;k)|\leq N^{k}$.
	
	The case of $P\in P_{L,\gamma}(u,v,c,k)$ is proved analogously.
	However, as there is no choice for the last vertex (which has to be $v$), we obtain $|P_{L,\gamma}(u,v,c,k)|\leq N^{k-1}$.
\end{proof}
The next observation teases the start of our random colouring procedure and bounds the probability of a suspicious path becoming dangerous.
\begin{observation}\label{obs:probability-dangerous}
	Let $\gamma$ be a partial $1$-edge-colouring of $G$ from some lists $L(e)\subset \cL(e)$.
	Let $0 \leq \p \leq 1$ and $L$ be an integer such that $|L(e) |= L$ for every edge $e$.
	Let  $\gamma'$ be an extension of $\gamma$ obtained as follows:
	\begin{enumerate}[\upshape(I)]
		\item   {Activate} each uncoloured edge with probability $\p$.
		\item Assign to each activated edge $e$ a colour chosen uniformly at random from $L(e)$.
	\end{enumerate}
	Then the probability that a suspicious path $P \in P_{L,\gamma}(u,v,c;k)$ is dangerous for $(u,v,c)$ under $\gamma'$ is $(\p/L)^{k}$.
\end{observation}
\begin{proof} 
	The probability that an edge $e$ is activated and assigned some colour $\alpha \in L(e)$ is $\p/L$.
	Since activations and colour assignments are independent and $P$ has uncoloured length $k$, the observation follows.
\end{proof}
In practise, we can assume that the parameters $N$ and $L$ of Observation~\ref{obs:number-suspicious} and~\ref{obs:probability-dangerous} satisfy $L \geq N$ and $\p$ is an absolute constant less than one.
Hence under the assumptions of Observation~\ref{obs:number-suspicious} and~\ref{obs:probability-dangerous}, the expected number of suspicious paths around an edge $uv$ under $\gamma$ that become dangerous under $\gamma'$ (defined as in Observation~\ref{obs:probability-dangerous}) is
\begin{align*}
	\sum_{c \in L(uv)} \sum_{k\geq 1} \left(\frac{\p}{ L}\right)^k  |P_{L,\gamma}(u,v,c;k)| \leq L \sum_{k\geq1} \left(\frac{\p}{ L}\right)^{k} N^{k-1} = \p \sum_{k\geq0} \left(\frac{\p N}{ L}\right)^k  = \frac{\p}{1-\p N/L} =O\left(1\right).
\end{align*}
The reason why we are using an activation probability $\p$ in our procedure is that it allows us to conveniently bound the term in the last step.
Figuratively speaking, the activations result in slowing down the iterative colouring procedure, allowing for more precision when preventing dangerous paths.
In any case, by Markov's inequality we can guarantee that with high probability, at most $O(1)$ dangerous paths emerge for the edge $uv$.
This term is dominated by the error terms of order $O({\sqrt{L}})$ arising from the (inevitable) concentration analysis elsewhere in the procedure.	
Hence, deleting the respective colours from $L(uv)$ and uncolouring $uv$ (if necessary) does not have a significant effect on the remainder of the analysis.

We remark that in the actual proof, a slight variation of this idea is used.
For the application of probabilistic tools such as  the Lov\'asz Local Lemma and Talagrand's inequality (see Section~\ref{sec:cycle-tweak} and~\ref{sec:uniform-bounds}), we will need to control the dependencies of events related to an edge $uv$.
For this purpose, it will be convenient to consider only suspicious paths $P(u,v,c;k)$ up to a bounded number of uncoloured edges $k < \ell$.
After this point, we simply keep track of the suspicious paths with $\ell$ uncoloured edges starting (but not necessarily ending) in $u$ or $v$, more precisely $P(u,c;\ell) \cup P(v,c;\ell)$.
This cut-off comes at the cost of a factor of $N$ in the number of suspicious paths (as seen in Observation~\ref{obs:number-suspicious}).
However, a straight forward calculation shows that the overall error is still bounded as above, provided that $\p/ {L} \geq \p^{2\ell}$.
(The details can be found in Section~\ref{sec:cycle-tweak}.)
Thus, in practise we can take $\ell = 2\log \Delta \geq 2\log N$ and $p =1/4$, where $\Delta$ is the maximum degree of $G$.
Having discussed the differences to the usual colouring procedure, let us continue with the proof of Theorem~\ref{thm:main}.

\subsection{Adapting former work}\label{sec:adapting-former-work}
In the following, we describe the details of the iterative random colouring procedure.
Apart from the above discussed modification, the proof is essentially the same as in the setting of ordinary $1$-edge-colourings.
Our exposition therefore follows closely the one presented in the textbook of Molloy and Reed~\cite{MR13} on this topic.

As a reminder, the iterative random colouring procedure is carried out as follows.
In each step,  pick a random colour from the list of (still available) colours for each uncoloured edge.
Then uncolour all arising conflicts.
It can be shown that with positive probability, this results in a  well-behaved partial $1$-edge-colouring allowing for the repetition of this process until most edges of the graph have been coloured.
Once this is done, we can complete the colouring with the following lemma of Alon~\cite{Alo93} using a fresh set of colours that was reserved beforehand.
(We will apply Lemma~\ref{lem:finisher} with the original lists $\cL'$. So trivially, no new monochromatic cycles are generated.)
\begin{lemma}[Finishing blow]\label{lem:finisher}
	Let $G$ be a graph with an assignment of lists $L(e)$ to the edges such that $|L(e)| \geq L$ for some constant $L$.
	Suppose that there is a constant $N$ such that for every edge $e$ and colour $c$, there are at most $N$ edges $f$ incident with $e$ such that $c \in L(f)$.
	If $L \geq 8 N$, then $G$ has a $1$-edge-colouring from the lists $L$.
\end{lemma}

In the rest of this section, we discuss the colour reservations (Lemma~\ref{lem:reserved-colours}), introduce a set of parameters to guide the iterative colouring process (Setup~\ref{setup} and Lemma~\ref{lem:sizes}) and formulate a claim that captures the status of the colouring before and after each step (Claim~\ref{cla:nibbler}).
Finally, we combine these results to prove Theorem~\ref{thm:main}.

\paragraph{Reservations.}
As mentioned above, we will set a few colours aside, which shall be used to finish the colouring later on.
For each vertex $v$, we select a set of colours $\Res(v)$ from the lists of edges incident to $v$.
The next lemma states that we can choose the sets $\Res(v)$ in a well-distributed way.
The proof is a standard application of concentration inequalities (Chernoff's bound).

\begin{lemma}[Reserve colours {\cite[Lemma 14.6]{MR13}}]\label{lem:reserved-colours}
	For $\Delta$ large enough the following holds.
	Let $G$ be a graph with an assignment of lists $\mathcal{L}$ such that $|\cL(e)| \geq \Delta+ 6\sqrt{\Delta} \log^4 \Delta.$
	Then, for each vertex $v$, we can choose $\Res(v) \subset \mathcal{L}(v)$ such that for every edge $e=uv$ and colour $c \in \Res(v)$:
	\begin{enumerate}[\upshape (a)]
		\item $|\cL(e) \cap (\Res(u) \cup \Res(v))| \leq 3 \sqrt{\Delta} \log^4 \Delta$,
		\item $|\cL(e) \cap (\Res(u) \cap \Res(v))| \geq  \log^8 \Delta/2$, and
		\item $|\{u \in N(v) \colon c \in \Res(u)\cap \mathcal{L}(uv)\}| \leq 2 \sqrt{\Delta} \log^4 \Delta$.
	\end{enumerate}
\end{lemma}

Given the sets $\Res(v)$, we define for each edge $e=uv$,
\begin{align}
L_0(e) &= \mathcal{L}(e) \sm (\Res(u) \cup \Res(v)) \text{ and} \label{equ:def-Le} \\ 
\Res(e)&= \mathcal{L}(e) \cap (\Res(u) \cap \Res(v)).\label{equ:def-Rservee}
\end{align} 
During the iterative colouring procedure, we colour $G$ only from colours of $L_0(e)$.
In the final step, we colour the remainder of the graph from the colours of $\Res(e)$ using Lemma~\ref{lem:finisher}.
To this end, we need to monitor the relation between uncoloured edges and the reserved colours.
\begin{definition}\label{def:Rvc-tracking}
	For a partial edge-colouring  $\gamma$ of $G$, we let $R_{\gamma}(v,c)$ be the set of uncoloured edges $uv$ with  $c \in \Res(u)$.
\end{definition}
This concludes the details for the reservation of colours.

\paragraph{Parameters.}
Next, we define a number of parameters used to track the sizes of the lists of (still available) colours and other objects throughout the iterative colouring procedure.
\begin{setup}\label{setup}
	Let $\p =1/4$ and $\l= 2 \log \Delta$ and define recursively:
	\begin{align*}
	L_0 &= \Delta+ {6\sqrt{\Delta} \log^4 \Delta},	 & L_{i+1} &= L_i \cdot \K_i^2  - \sqrt{L_i} \log^2 \Delta,
		\\ N_0 &= \Delta, & N_{i+1} &= N_i \cdot  \K_i  \cdot \left(1- \p \R_i^2   \right) + \sqrt{N_i}  \log^2 \Delta,
		\\R_0 &= 2\sqrt{\Delta} \log^4 \Delta, & R_{i+1} &= R_i \cdot \left(1- \p \R_i^2   \right)   + \sqrt{R_i}\log^2 \Delta, 
		\\ 	\R_i &= \left(1-{\frac{\p}{L_i}}\right)^{N_i-1}, & \K_i &=  1- \p  \frac{N_i}{L_i} \cdot \R_i^2 \text{ and}
		\\ i_0 &= \min \{i \colon L_i < 3 \log^7 \Delta \}.
	\end{align*}
\end{setup}
We note that in the final procedure $\R_i^2$ will be the probability that an edge {retains} an assigned colour and $\K_i^2$ will be the probability that the list of an edge {keeps} a particular colour.
The following lemma follows from a  straightforward calculation. We omit the proof. 
\begin{lemma}[{Molloy and Reed~\cite[Lemma 7]{MR00}}]\label{lem:sizes}
	Suppose that $\Delta$ is sufficiency large and consider $R_{i}, L_{i}, N_{i},i_0$ as in Setup~\ref{setup}.
	Then it follows that $L_{i_0}, N_{i_0}, R_{i_0}  > \log^7 \Delta$, $R_{i_0} \leq 3 \log^{7.5} \Delta$, $R_{i}/L_{i} \leq \log \Delta$, and $L_i > N_i > L_i/2$ for each $i \leq i_0$.
\end{lemma}
We remark that the parameters $R_i$ do not appear in the original statement, but are added (under the name of $T_i$) later on in the discussion on list colouring~\cite[Section 5]{MR00}.
In a similar vain, the inequality $R_{i}/L_{i} \leq \log \Delta$ also does not appear in the original statement, but easily follows from the proof.
(See Appendix~\ref{sec:sizes-R/L} for further elaboration.)
Moreover, Molloy and Reed computed these bounds for $\p=1$, but the same proof and its conclusion hold for $\p =1/4$.
 
\paragraph{Iterations.}
During the iterative colouring procedure, we will track the (shrinking) lists of colours and colouring itself in order to ensure that no dangerous paths arise.
The next claim presents the heart of the procedure.

The following definitions help us with the bookkeeping.
A partial $1$-edge-colouring $\gamma$ of $G$ from $\cL(e)$ is called \emph{twin-acyclic} if it does not induce any cycle coloured (alternately) with twins.
We say that lists $L(e) \subset \cL(e)$ are \emph{$\gamma$-compatible} if both of the following hold:
\begin{itemize}
	\item if $\gamma(f)=c$, then $c \notin L(e)$ for all incident edges $e$ and $f$,
	\item if an edge $e=uv$ is not coloured by $\gamma$ and there exists a dangerous path for $(u,v,c)$ under $\gamma$, then $c\not\in L(e)$.
\end{itemize}
Another edge-colouring $\gamma'$ of $G$ \emph{extends} $\gamma$, if every edge coloured by $\gamma$ is coloured by $\gamma'$ in the same way.
Similarly, the lists $L(e)$ \emph{extend} lists $L'(e) \subset \cL(e)$, if $L(e) \subset L'(e)$ for all edges $e$ of $G$.

\begin{claim}[Single colouring step]\label{cla:nibbler}
	Given Setup~\ref{setup} and $\Delta$ large enough, the following holds for every $0 \leq i < i_0$.
	Suppose that there is a twin-acyclic $1$-edge-colouring $\gamma_i$ of $G$ and $\gamma_i$-compatible lists of colours $L_i(e) \subset \cL(e)$ with the following properties:
	For every uncoloured edge $e=uv$ and colour $c \in L_i(e)$, we have
	\begin{enumerate}[\upshape (a)]
		\item \label{itm:list-size-before}  $|L_i(e)| = L_i$, and
		\item \label{itm:colour-neighbours-size-before}  $|N_{G,L_i,\gamma_i}(v,c)| \leq N_i$, and
		\item \label{itm:reserved-size-before}  $|R_{\gamma_i}(v,c)| \leq R_i$, and
		\item \label{itm:cycle-conflict-before} there is no dangerous path for $(u,v,c)$ under $\gamma_i$.
	\end{enumerate}
	
	Then it holds that: there exist a twin-acyclic $1$-edge-colouring $\gamma_{i+1}$ of $G$ from $\cL(e)$ and $\gamma_{i+1}$-compatible lists of colours $L_{i+1}(e)$ such that $L_{i+1}(e),\gamma_{i+1}$  {extend} $L_{i}(e),\gamma_{i}$.
	Moreover, for every uncoloured edge $e=uv$ and colour $c \in L_{i+1}(e)$, we have 
	\begin{enumerate}[\upshape (a$'$)]
		\item \label{itm:list-size-after} $|L_{i+1}(e)| = L_{i+1}$, 
		\item $|N_{G,L_{i+1},\gamma_{i+1}}(v,c)| \leq N_{i+1}$, 
		\item \label{itm:reserved-size-after} $|R_{\gamma_{i+1}}(v,c)| \leq R_{i+1}$, and
		\item \label{itm:cycle-conflict-after} there is no dangerous path for $(u,v,c)$ under  $\gamma_{i+1}$.
	\end{enumerate}
\end{claim}
We prove Claim~\ref{cla:nibbler} in Section~\ref{sec:nibbler}.
For now, let us finish the proof of Theorem~\ref{thm:main}.

\paragraph{Finishing the proof.}
We iteratively apply Claim~\ref{cla:nibbler} starting with $L_0(e)$, defined in~\eqref{equ:def-Le}, and the empty colouring of $G$, denoted by $\gamma_0$.
This yields a sequences of extensions $(L_i,\gamma_i)$ satisfying the outcome of Claim~\ref{cla:nibbler}.
In particular, each $\gamma_{i}$ is a $1$-edge-colouring, which does not contain any monochromatic cycle in terms of $\cL'$.
After $i_0$ steps, Lemma~\ref{lem:sizes} yields that for every edge $e$ and colour $c \in \Res(e)$, defined in~\eqref{equ:def-Rservee}, there are at most $2R_{i_0} \le 6   \log^{7.5} \Delta$ edges $f$ incident to $e$ for which $c \in \Res(f)$.
On the other hand $|\Res(e)| \geq  (\log^8 \Delta)/2$ by Lemma~\ref{lem:reserved-colours}.
Denote the transposition of $\Res(e)$ to $\cL'$ by $$\Res'(e)=\{c\colon \text{$(c,1)$ or $(c,2) \in \Res(e)$}\}.$$
Thus $|\Res'(e)| \ge (\log^8 \Delta)/4$ and yet for every edge $e$ and colour $c \in \Res'(e)$, there are at most $4R_{i_0} \le 12   \log^{7.5} \Delta$ edges $f$ incident to $e$ for which $c \in \Res'(f)$. We then finish by colouring the remaining uncoloured edges of $G$ from the lists $\Res'(e)$ using Lemma~\ref{lem:finisher}.
Since this is a $1$-edge-colouring, no new monochromatic cycles are generated.
Hence we have proved Theorem~\ref{thm:main}.

\section{Proof of Claim~\ref{cla:nibbler}}\label{sec:nibbler}
In this section we use a random colouring procedure to prove Claim~\ref{cla:nibbler}.
To state the procedure, we introduce a few further definitions.
For convenience, we denote the lists by $L(e)=L_i(e)$, colour neighbours by $N(v,c)=N_{G,L_i,\gamma_i}(v,c)$ and $R(v,c)=R_{\gamma_i}(v,c)$.
The random colouring procedure involves two coin flips that will help us to bound some of the involved terms uniformly.
Recall the definitions of Setup~\ref{setup}.
For an edge $e$, a vertex $v$ and a colour $c$, let
\begin{align}
\Eq(e,c) &=  1- \frac{\R_i^2}{ \left(1-  \frac{\p}{L_i} \right)^{| N(u,c)|+| N(v,c)|}}  \text{ and } \label{equ:Eq} \\
\Vq(v,c) &= 1-\frac{\K_i}{1- \frac{\p}{L_i} |N(v,c)| \R_i^2}. \label{equ:Vq}
\end{align}
Note that by assumption~\ref{itm:colour-neighbours-size-before} of Claim~\ref{cla:nibbler}, we have $|N(u,c)|,  |N(v,c)| \leq N_i$.
Moreover, $L_i > N_i$ by Lemma~\ref{lem:sizes} and as $i <i_0$.
Hence, it follows that $0 \le \Eq(e,c),\Vq(v,c) \leq 1$, which qualifies these terms to play the role of a coin flip probability.

Now we colour (some of) the edges of $G$ using the following random process:
\begin{procedure*}\label{procedure}~
	\begin{enumerate}[(I)]
		\item \label{step:activate} \textbf{Edge activation.} {Activate} each uncoloured edge with probability $\p$.
		
		\item \label{step:assign} \textbf{Assign colours.} Assign to each activated edge $e$ a colour chosen uniformly at random from $L(e)$.
		
		\item \label{step:unassign-and-coinflip} \textbf{Resolve conflicts.}
		Uncolour every edge $e$, which is assigned the same colour as an incident edge.
		If $e$ is assigned colour $c$ and no neighbour was assigned $e$, then uncolour $e$ with probability $\Eq(e,c)$. 
		\item \label{step:update-lists-and-coinflip}  \textbf{Update lists.}
		For every vertex $v$ and colour $c$, if $c$ is retained by some edge in $N(v,c)$, then remove $c$ from the lists $L(f)$ of all other edges $f \in N(v,c)$.
		If $c$ is not retained by an edge in $N(v,c)$, then with probability $\Vq(v,c)$ remove $c$ from the lists $L(f)$ of all edges $f \in N(v,c)$.
		
	\end{enumerate}
\end{procedure*}
Let us denote the partial $1$-edge-colouring and lists of colours after step~\ref{step:update-lists-and-coinflip} of the procedure by  $\gamma'$ and $L'(e)$.
We also abbreviate $N'(v,c)=N_{G,L',\gamma'}(v,c)$ and $R'(v,c)=R_{\gamma'}(v,c)$.
A concentration analysis shows that the sizes of the random variables $L'(e)$, $N'(v,c)$ and $R'(v,c)$ are (individually) bounded with high probability.
The proof of Claim~\ref{cla:MR-edge-bounds} is fairly standard by now (see for instance the monograph of Molloy and Reed~\cite{MR13}).
For sake of completeness, the details of the argument are spelled out in Appendix~\ref{sec:MR-edge-bounds}.

\begin{claim}[{Molloy and Reed~\cite[Claim 14.9--14]{MR13}}]\label{cla:MR-edge-bounds}
	Fix an edge $e$, a vertex $v$ and a colour $c$.
	It holds with probability at least $1-\Delta^{-10\log \Delta}$ that
	\begin{align*}
		|L'(e)| &\geq L_i \cdot \K_i^2 - \tfrac{1}{2} \sqrt{L_i} \log^2 \Delta,
		\\|N'(v,c)| &\leq N_i \cdot  \K_i  \cdot \left(1- \p \R_i^2   \right) + \tfrac{1}{2}\sqrt{N_i} \log^2 \Delta \text{, and}
		\\ |R'(v,c)| &\leq R_i \cdot  \K_i + \tfrac{1}{2} \sqrt{R_i} \log^2 \Delta.
	\end{align*}
\end{claim}

Note that we could finish with an application of the Lov\'asz Local Lemma at this point and obtain lists and a $1$-edge-colouring that would satisfy~\ref{itm:list-size-after}--\ref{itm:reserved-size-after} of Claim~\ref{cla:nibbler}.
If our goal was to obtain a (partial) $2$-edge-colouring of $G$ from $\cL'$, then this would be the end of the proof.
However, we also want to suppress dangerous paths, as formalized in~\ref{itm:cycle-conflict-after}.

Consider an edge $e=uv$ and colour $c$ such that there is a dangerous path $P$ for $(u,v,c)$ under $\gamma'$.
By condition~\ref{itm:cycle-conflict-before} of Claim~\ref{cla:nibbler}, $P$ must have an edge that was uncoloured under $\gamma$, which without loss of generality we can assume to be $e$.
Hence  either  there is some $1 \leq k <\ell$, where $\ell$  was defined in Setup~\ref{setup}, such that $P \in P(u,v,c;k)$ or $P$ starts with a segment of $P(u,c;\ell) \cup P(v,c;\ell)$.
We can therefore focus on dangerous paths that correspond to 
\begin{align}\label{equ:dangerous-paths}
	P(u,v,c) := \bigcup_{k=1}^{\ell-1} P(u,v,c;k) \cup P(u,c;\ell) \cup P(v,c;\ell).
\end{align}
The following modification, which we carry out after step~\ref{step:update-lists-and-coinflip}, deals with these paths.
\begin{enumerate}[(I)] \addtocounter{enumi}{4}
	\item \label{step:cycle-tweak}  \textbf{Prevent dangerous paths.}  For every edge $uv$ uncoloured in $\gamma$ and colour $c \in L(uv)$, if  there is a path in $P(u,v,c)$ that is dangerous for $(u,v,c)$ after step~\ref{step:assign}, then remove $c$ from $L(uv)$. 
	If $e$ was assigned $c$ also uncolour $uv$.
	(If $c$ was already removed in step~\ref{step:unassign-and-coinflip} or~\ref{step:update-lists-and-coinflip}, do nothing.)
\end{enumerate}
We remark that step~\ref{step:cycle-tweak} is wasteful since colours are removed from the lists on basis of assignments and not retainments as in step~\ref{step:update-lists-and-coinflip}.
While this increases the error term slightly, it also simplifies the analysis in the proof of Claim~\ref{cla:cycle-tweak} below.

We denote the colouring after step~\ref{step:cycle-tweak} by $\gamma''$ and the lists of colours by $L''$.
Clearly $(\gamma'',L'')$ is an extension of $(\gamma,L)$.
Moreover, $\gamma''$ is still a $1$-edge-colouring and also satisfies~\ref{itm:cycle-conflict-after} of Claim~\ref{cla:nibbler}.
It remains to show that $\gamma''$ satisfies~\ref{itm:list-size-after}--\ref{itm:reserved-size-after}.
To this end, we need to now bound the additional error terms related to step~\ref{step:cycle-tweak}.
The following random variables help us to track the deviations.

\begin{definition}
	Consider an (in $\gamma$) uncoloured edge $uv$ and a colour $c$.
	\begin{itemize}
		\item Let $X(uv)$ be the set of colours $c\in L(uv)$ for which there is a path in $P(u,v,c)$ that is dangerous for $(u,v,c)$ after step~\ref{step:assign}.
		\item Let $Y(v,c)$ be the set of edges $vw \in N(v,c)$ for which there is a path in $P(v,w,c)$ that is dangerous for $(v,w,c)$ after step~\ref{step:assign}.
		\item Let $Z(v,c)$ be the set of edges $vw$ with  $c \in \Res(w)$, for which there is a path in $P(v,w,c)$ that is dangerous for $(v,w,c)$ after step~\ref{step:assign}.
	\end{itemize}
\end{definition}
Note that
\begin{align*}
|L''(e)| &\geq |L'(e)| - |X(e)|,
\\	|N''(v,c)| &\leq |N'(v,c)| + |Y(v,c)|, \text{ and }
\\|	R''(v,c) |&\leq |R'(v,c)| + |Z(v,c)|.
\end{align*}
The next claim shows that the sizes of $X(e)$, $Y(v,c)$ and $Z(v,c)$ are bounded with high probability.
\begin{claim}\label{cla:cycle-tweak}
	Fix an edge $e$, a vertex $v$ and a colour $c$.
	It holds with probability at least $1-\Delta^{-10\log \Delta}$ that 
	\begin{align*}
		|X(e)| &\leq  \tfrac{1}{4} \sqrt{L_i} \log^2 \Delta,
		\\ |Y(v,c)| &\leq \tfrac{1}{4} \sqrt{N_i} \log^2 \Delta \text{, and}
		\\ |Z(v,c)| &\leq \tfrac{1}{4} \sqrt{R_i} \log^2 \Delta.
	\end{align*}
\end{claim}
We give a proof of Claim~\ref{cla:cycle-tweak} in Section~\ref{sec:cycle-tweak}.
Our final claim shows that the sizes of the above defined random variables can be bounded simultaneously.

\begin{claim}\label{cla:simultaneously}
	With positive probability, the conclusions of Claim~\ref{cla:MR-edge-bounds} and~\ref{cla:cycle-tweak} hold for all edges $e$, vertices $v$ and a colours $c$.
\end{claim}
The proof of  Claim~\ref{cla:simultaneously} can be found in Section~\ref{sec:uniform-bounds}.
For now, we continue with the proof of Claim~\ref{cla:nibbler}.
Denote by  $N''(v,c)$, $R''(v,c)$ the objects obtained from $N'(v,c)$, $R'(v,c)$ after step~\ref{step:cycle-tweak}.
By Claim~\ref{cla:simultaneously}, we can choose the colour assignments such that
\begin{align*}
|L''(e)| &\geq \left( L_i \cdot \K_i^2 - \tfrac{1}{2} \sqrt{L_i} \log^2 \Delta\right) - \left(\tfrac{1}{2} \sqrt{L_i} \log^2 \Delta\right) = L_{i+1},
\\	|N''(v,c)| &\leq   \left(N_i \cdot  \K_i  \cdot \left(1- \p \R_i^2   \right) + \tfrac{1}{2}\sqrt{N_i} \log^2 \Delta \right)+ \left(\tfrac{1}{2} \sqrt{N_i} \log^2 \Delta\right) = N_{i+1}, \text{ and }
\\|R''(v,c)| & \leq \left(R_i \cdot  \K_i + \tfrac{1}{2} \sqrt{R_i} \log^2 \Delta \right) + \left(\tfrac{1}{2} \sqrt{R_i} \log^2 \Delta\right) = R_{i+1}.
\end{align*}
Take $L_{i+1}(e)=L''(e)$ and $\gamma_{i+1} = \gamma''$.
Note that $\gamma_{i+1}$ is twin-acyclic and the lists $L_{i+1}(e)$ are $\gamma_{i+1}$-compatible by Claims~\ref{cla:MR-edge-bounds} and~\ref{cla:cycle-tweak} due to steps~\ref{step:update-lists-and-coinflip} and~\ref{step:cycle-tweak} of the procedure.
Lastly, to ensure that $|L_{i+1}(e)|=L_{i+1}$, we delete $|L_{i+1}(e)|-L_{i+1}$ colours from every list $L_{i+1}(e)$ with $|L''(e)| > L_{i+1}$.
This finishes the proof of Claim~\ref{cla:nibbler}.
 
\section{Cycle modification analysis}\label{sec:cycle-tweak}
In this section, we prove Claim~\ref{cla:cycle-tweak}.
Our approach relies on a straightforward concentration analysis.
First we show that the expectation of $|X(e)|$ and $|Y(v,c)|$ are bounded by $4\p$ and the expectation of $|Z(v,c)|$ is bounded by $4\p \log \Delta$.
Then we show that with high probability, neither of those random variables deviates by more than $\tfrac{1}{2}\sqrt{L_i} \log^2 \Delta$ ($\sqrt{N_i}$, $\sqrt{R_i}$ respectively) from their expectation.

\subsection{Expectation.}
Let us start with the expectation of $X(e)$ for an edge $e=uv$.
As a reminder, this is the set of colours $c\in L(uv)$ for which there is a path in $P(u,v,c)$ that is dangerous for $(u,v,c)$ after step~\ref{step:assign}.
By~\eqref{equ:dangerous-paths}, such a path $P$ is in  $P(u,c;\ell)$, $P(v,c;\ell)$ or $P(u,v,c;k)$ for some $1 \leq k \leq \ell-1$.
Since $\gamma$ is a $1$-edge-colouring and by~\ref{itm:colour-neighbours-size-before} of Claim~\ref{cla:nibbler}, we can use Observation~\ref{obs:number-suspicious} to  bound $|P(u,c;\ell)|$, $|P(v,c;\ell)| \leq N_i^{\ell}$ and $|P(u,v,c;k)| \leq N_i^{k-1}$   for every $1 \leq k \leq \ell-1$.
Moreover, the probability that $P \in P(u,v,c;k)$ is dangerous for $(u,v,c)$ after step~\ref{step:assign} is $(\p/ L_i)^{k}$.
Similarly, the probability that $P \in P(u,c;\ell)$ is dangerous for $(u,v,c)$ after step~\ref{step:assign} is $(\p/ L_i)^{\ell}$ by Observation~\ref{obs:probability-dangerous}.
Also note that by choice of $p=1/4$ and $\ell=2\log \Delta$ in Setup~\ref{setup}, we have
\begin{align*}
	 \p^{\l} = \p\cdot \frac{1}{4^{2 \log \Delta -1}}  \leq \frac{\p}{2\Delta} \leq  \frac{\p}{L_i}.
\end{align*}
Together,  we can bound the probability that a path in $P(u,v,c)$ is dangerous for $(u,v,c)$ after step~\ref{step:assign} by
\begin{align*}
 2 \left(\frac{\p N_i}{L_i}  \right)^{\l} +    \frac{\p}{L_i} \sum_{k=0}^{{\l}-1} \left(\frac{\p N_i}{L_i}\right)^k 
&\leq 2\p^{{\l}} +  \frac{\p}{L_i} \cdot \frac{ 1}{1-\frac{\p N_i}{L_i}}  
  \leq \frac{4\p}{L_i},
\end{align*}
where we used in the penultimate inequality that $L_i\geq N_i$, which holds by Lemma~\ref{lem:sizes}.
Linearity of expectation then gives  $\Exp(|X(e)|) \leq 4\p$.
The expectations of $Y(v,c)$ and $Z(v,c)$ follow analogously, since $|N(v,c)| \leq N_i$, $|R(v,c)| \leq R_i$ by assumption and $N_i \leq L_i$ and $R_i / L_i \leq \log \Delta$ by Lemma~\ref{lem:sizes}.

\subsection{Concentration.}
We use Talagrand's inequality to show that the random variables $X,Y,Z$ are highly concentrated around their expectation. 
\begin{theorem}[Talagrand's inequality~\cite{Tal95}]\label{thm:talagrand}
	Let $X$ be a non-negative random variable determined by the independent trials $T_1,\ldots,T_n$. Suppose that for every set of possible outcomes of the trials, we have:
	\begin{enumerate}[(i)]
		\item changing the outcome of any one trial can affect $X$ by at most $c$ and
		\item for each $s > 0$, if $X \geq s$ then there is a set of at most $rs$ trials whose outcomes certify $X \geq s$.
	\end{enumerate}
	Then for any $t > 96c \sqrt{r \Exp(X)} + 128rc^2$ we have
	\begin{align*}
	\Pr(|X-\Exp(X)| > t) \leq 4 \exp\left({-\frac{t^2}{8c^2r(4\Exp(X)+t)}}\right).
	\end{align*}
\end{theorem}
 
Fix an edge $e=uv$ and let us abbreviate $X = |X(e)|$.
Our intention is to apply Talagrand's inequality to $X$ with $c=2$, $r=\l$ and $t=\tfrac{1}{4}\sqrt{L_i} \log^2\Delta$.
Note that changing any the outcome of the edge activation in step~\ref{step:activate} or colour assignment in step~\ref{step:assign} can affect the size of $X$ by at most $2$.
This is because at most one colour might be added to or removed from $X(e)$ this way.
Moreover, if $X \geq s$, then for every colour $c \in X(e)$, there must be a path in $P \in   P(u,v,c)$ together with at most $ \l$ edges of $P$ that have been activated and assigned either $c$ or its twin colour.
(Here it becomes apparent why we chose to restrict the paths in $P(u,v,c)$ to $\ell$ edges.)
Recall that by assumption as $i < i_0$, we have that $L_i \geq \log^7 \Delta$.
Moreover, $\Exp(X) \leq 4\p$ by the above.
Since $\Delta$ is large enough, we have
$$t=\tfrac{1}{2}\sqrt{L_i} \log^2\Delta \geq \log^5 \Delta>192 \sqrt{4\p\l} + 512\l \geq 96c \sqrt{r \Exp(X)} + 128rc^2.$$
Hence we obtain from Talagrand's inequality (Theorem~\ref{thm:talagrand}) that 
\begin{align*}
	\Pr(|X-\Exp(X)| > t) \leq 4 \exp\left({-\frac{t^2}{8c^2r(4\Exp(X)+t)}}\right) \leq   \exp\left(-\frac{\sqrt{L_i}\log \Delta}{2^9}\right) \leq \Delta^{-10\log \Delta},
\end{align*}
where we used $4\Exp(X)\leq t$ and $r=\ell = 2 \log \Delta$ in the second inequality.
The concentration of $Y(v,c)$ and $R(v,c)$ follows along the same lines.

\section{Uniform bounds}\label{sec:uniform-bounds}
In this section, we prove Claim~\ref{cla:simultaneously}.
We require the Lov\'asz Local Lemma~\cite{EL75}.

\begin{lemma}[{Lov\'asz Local Lemma}]\label{lem:LLL}
		Let $A_1,\dots,A_n$ be events in an arbitrary probability space.
		Suppose that each event $A_i$ is mutually independent of all but at most $d$ other events, and that $\Pr(A_i) \leq p$ for all $1 \leq i \leq n$.
		If $e p (d+1) \leq 1$, then $\Pr\left( \bigwedge_{i=1}^n \overline{A_i} \right) >0$.
\end{lemma}

Recall the definition of suspicious paths (Definition~\ref{def:suspicious}).
We say that a path $P$ is  {\emph{suspicious} for a colour $c$} if there are vertices $u$, $v$ such that $P$ is suspicious for $(u,v,c)$.
The following observation follows in the same fashion as Observation~\ref{obs:number-suspicious} (using that $\gamma$ is a $1$-edge-colouring and conditions~\ref{itm:list-size-before},~\ref{itm:colour-neighbours-size-before} of Claim~\ref{cla:nibbler}).
We omit the proof.
\begin{observation}\label{obs:suspicous-path-for-c}
	For every uncoloured edge $e$, there are at most $L_iN_i^\ell$ paths of length at most $\ell$ which start with $e$ and are suspicious for some colour $c \in L(e)$.
\end{observation}

Finally, we prove Claim~\ref{cla:simultaneously}.
\begin{proof}[Proof of Claim~\ref{cla:simultaneously}]
	We wish to apply the  Lov\'asz Local Lemma.
	To this end, let us discuss the dependencies between the random variables  $L'(e)$ , $N'(v,c)$, $R'(v,c)$,  $X(e)$,  $Y(v,c)$ and $Z(v,c)$.
	
	The variables $L'(e)$ , $N'(v,c)$ and $R'(v,c)$ are determined by the activations and colour assignments to edges of distance at most 2 to $v$ or $e$ in the graph $G$.
	Hence random variables of type $L'(e)$ , $N'(v,c)$ and $R'(v,c)$ are each independent of all but at most $(2N_iL_iR_i)^2$ random variables of type $L'(\cdot)$ , $N'(\cdot,\cdot)$ and $R'(\cdot,\cdot)$.

	The dependency relations associated with random variables of type $X(e)$, $Y(v,c)$ and $Z(v,c)$ are a bit more delicate.
	This is because the events determining these variables are not constrained by graph distance.

	Consider an uncoloured edge $e$.
	In the following, we bound the number of random variables of type $X(\cdot)$ that $X(e)$ is not independent of.
	Let $f$ be another uncoloured edge.
	Observe that the random variables $X(e)$ and $X(f)$ are independent unless the following holds.
	There are suspicious paths $P,P'$ for colours $c \in L(e)$, $c'\in L(f)$ that start with $e,f$, respectively.
	Moreover, $P$ and $P'$ have each at most $\ell$ uncoloured edges and cross in an uncoloured edge  $g$.
	In light of Observation~\ref{obs:suspicous-path-for-c}, it follows that $X(e)$ is independent of all but at most $(L_iN_i^\ell)^2$ random variables of type  $X(f)$.
	Note that the first factor $L_iN_i^\ell$ comes from counting the suspicious paths starting at $e$, while the second factor comes from counting the suspicious paths starting at $g$ (as described before).
	
	In the same way, we can argue that $X(e)$ is independent of all but at most $(L_iN_i^\ell)(L_iN_i^{\ell+1})$ random variables of type $Y(v,c)$.
	Note that here, the second factor counts the at most $L_iN^{\ell}$ suspicious paths ending in one of the at most $N_i$ edges of $N(v,c)$.
	We also obtain independence for $X(e)$ of all but at most $(L_iN_i^\ell)(L_iN_i^{\ell}R_i)$ random variables of type  $Z(v,c)$.
	Finally, we see that $X(e)$ is independent of all but at most $(L_iN_i^\ell)(2L_i^2)$, $(L_iN_i^\ell)(2N_i^2)$ and $(L_iN_i^\ell)(2R_i^2)$  random variables $L'(f)$, $N'(v,c)$ and $R'(v,c)$, respectively.
	This concludes the analysis of the dependency relations for the random variable $X(e)$.
	
	Analogous statements can be made from the perspective of random variables $Y(v,c)$, $Z(v,c)$, $L'(e)$, $N'(v,c)$ and $R'(v,c)$.
	Since the arguments are identical, we omit the details.
	In summary, each of the discussed random variables is independent of all but at most 
	$$d= 6(2L_i N_i^{\ell+1} R_i)^2 \leq \Delta^{4\ell} = \Delta^{8\log \Delta} \ll \Delta^{10\log \Delta}$$ 
	of the other random variables.
	Hence we can apply the  Lov\'asz Local Lemma (Lemma~\ref{lem:LLL}) with $d$ and $p=\Delta^{-10\log \Delta}$. 
	It follows that with positive probability the outcome satisfies the conclusions of Claim~\ref{cla:MR-edge-bounds}  and~\ref{cla:cycle-tweak} hold simultaneously for all edges $e$, vertices $v$ and colours $c$.
\end{proof}

\section*{Acknowledgements}
We thank the referees for carefully reading the manuscript, which has lead to an improved presentation of the proofs.
 
\bibliographystyle{amsplain}
\bibliography{bibliography}

\newpage
\appendix

\section{Proof of Lemma~\ref{lem:sizes}}\label{sec:sizes-R/L}

The statement of Lemma~\ref{lem:sizes} has been slightly extended by adding the bound on the ratio $R_i/L_i$.
We remark that this relation is easily implied by the proof of Molloy and Reed~\cite{MR00}.
For completeness, we present the details.

\begin{proof}[Proof of $R_i/L_i \leq \log \Delta$ in Lemma~\ref{lem:sizes}]
	Let us start by setting up `untainted' versions of the constants $R_i$, $L_i$ and $\K_i$.
	(Note that $\R_i$ is used in its old state.)
	As in Setup~\ref{setup}, let $\p=1/4$.
	We then recursively define
	\begin{align*}
	L_0^* &= \Delta,	 & L_{i+1}^* &= L_i^* \cdot (\K_i^*)^2,
	\\ N_0^* &= \Delta, & N_{i+1}^* &= N_i^* \cdot  \K_i^*  \cdot \left(1- \p \R_i^2   \right),
	\\R_0^* &=  2\sqrt{\Delta} \log^4 \Delta  , & R_{i+1}^* &= R_i^* \cdot \left(1- \p \R_i^2   \right), 
	\\ 	  &  & \K_i^* &=  1- \p  \frac{N_i^*}{L_i^*} \cdot \R_i^2.
	\end{align*}
	A straightforward induction, yields that $L_i^* = N_i^*$ and $\K_i^*=1-p\R_i^2$ for all $i \geq 1$.
	Therefore, 
	\begin{align*} 
		R_i^* =  2\sqrt{\Delta} \log^4 \Delta \cdot \prod_{j=1}^{i-1} (1- p\R_j^2) \qquad \text{ and }  \qquad L_i^* = N_i^* = \Delta \prod_{j=1}^{i-1} (1- p\R_j^2)^2.
	\end{align*}
	Hence, for $L_i^* > \log^7 \Delta$, we have 
	\begin{align*} 
		\frac{R_i^*}{L_i^*}  = \frac{2 \log^4 \Delta}{\sqrt{\Delta} \prod_{j=1}^{i-1} (1- p\R_j^2)} < 2 \sqrt{\log  \Delta}.   
	\end{align*}
	It can be shown that the $R_i,L_i$ are not very far from  $R_i^*,L_i^*$ respectively.
	More precisely, for every $1 \leq i \leq i_0$, it holds that
	\begin{align*} 
		R_i \leq R_i^* + \sqrt{R_i^*} \log^{2.5} \Delta\qquad \text{ and }  \qquad L_i \geq L_i^* - \sqrt{L_i^*} \log^{2.5} \Delta.
	\end{align*}
	
	Let us prove the first statement, since its prove is only given implicitly in the paper of Molloy and Reed~\cite{MR00}.
	Assuming that the claim holds for $i < i_0$ and $\Delta$ is sufficiently large, we have 
	\begin{align*}
		R_{i+1} &= R_{i} \cdot \left(1- \p \R_{i}^2   \right) + R_{i} \log^2 \Delta \\
		 		&\leq (R_i^* + \sqrt{R_i^*} \log^{2.5} \Delta)    \cdot \left(1- \p \R_{i}^2   \right) + R_{i} \log^2 \Delta \\
		 		&\leq R_{i+1}^* +  \sqrt{R_{i+1}^*} \log^{2.5} \Delta  \cdot \sqrt{1- \p \R_{i}^2 } + R_{i} \log^2 \Delta \\
		 		&\leq R_{i+1}^* +  \sqrt{R_{i+1}^*} \log^{2.5} \Delta .
	\end{align*}
	
	Given this, we can easily bound the ratio.
	More precisely, while $L_i^*, R^*_i > \log^7 \Delta$, it follows that
	\begin{align*} 
	  		\frac{R_i}{L_i}  \leq  	2\frac{R_i^*}{L_i^*}  <   4  \sqrt{\log  \Delta} < \log \Delta.  
	\end{align*}
\end{proof}

\section{Proof of Claim~\ref{cla:MR-edge-bounds}}\label{sec:MR-edge-bounds}
The proof of Claim~\ref{cla:MR-edge-bounds} can be found in the monograph of Molloy and Reed~\cite{MR13}.
The only difference in our setting is that we use an activation probability and slightly different constants.
However, this does not require any meaningful change in the steps in the proof.
The following arguments are therefore almost word by word identical to the ones of Molloy and Reed.

We prove  Claim~\ref{cla:MR-edge-bounds} using a concentration analysis.
To this end, we first bound the expectation of the random variables in question (Subsection~\ref{sec:MR-expectation}) and then show that these random variables are highly concentrated around their expectation (Subsection~\ref{sec:MR-concentration}) .
\subsection{Expectation} 
\label{sec:MR-expectation}
The following three claims bound the expectations of $R'(v,c)$, $L'(e)$ and $N'(v,c)$.

\begin{claim}[{\cite[Claim 14.9]{MR13}}]
	We have	$\Exp(|R'(v,c)|)  \leq  R_i \cdot  \K_i$ for every vertex $v$ and colour $c$.
\end{claim}
\begin{proof}
	Fix an edge $e$, a vertex $v$ and a colour $c$.
	Suppose that $e$ is assigned colour $c$ in step~\ref{step:assign}  of the procedure.
	Recall the definition of $\R_i$, $\K_i$ in  Setup~\ref{setup}  and $\Eq$ in~\eqref{equ:Eq}.
	It follows that the probability that $e$ retains $c$ after step~\ref{step:unassign-and-coinflip} is
	\begin{align*}
	\left(1-  \frac{\p}{L_i} \right)^{| N(u,c)|+| N(v,c)|-2} \cdot (1-\Eq(e,c)) =   \R_i^2,
	\end{align*}
	where each factor $1-  \frac{\p}{L_i}$ represents the probability that one of the colour neighbours of $e$ is assigned $c$.
	(The exponent of $-2$ comes from the fact that we do not count the edge $e$ in either case.)
	Hence, if $e$ is uncoloured at the beginning, then the probability that $e$ is coloured after step~\ref{step:update-lists-and-coinflip} is
	\begin{align*}
	 1-  \p     \R_i^2  \leq 1- \p\frac{N_i}{L_i}   \R_i^2= \K_i,
	\end{align*}
	where $\p$ is the probability that $e$ is activated.
	Given this, the claim follows by linearity of expectation.
\end{proof}

\begin{claim}[{\cite[Claim 14.10]{MR13}}]
	We have	$\Exp(|L'(e)|) \geq L_i \cdot \K_i^2 $ for every edge $e$ and colour $c$.
\end{claim}
\begin{proof}
	Consider an edge $e = uv$ and colour $c \in L(e)$. 
	We show that the probability that $c \in L_i(e)$ is at most $\K_i^2$.
	From this, the claim follows again by linearity of expectation.
	
	Let $K_u$, $K_v$ be the events that no edge in $N(u,c) \sm e$, respectively $N(v,c) \sm e$
	retains $c$ after step~\ref{step:unassign-and-coinflip} of the procedure. 
	It turns out that the simplest way to compute $\Pr(K_u \cap K_v)$ is
	through the indirect route of computing $\Pr(\overline{K_u} \cup \overline{K_v})$ which is equal to
	$1 - \Pr(K_u \cap K_v)$. 
	Now, by the most basic case of the Inclusion-Exclusion Principle, 
	$\Pr(\overline{K_u} \cup \overline{K_v})$ is equal to $\Pr(\overline{K_u})+ \Pr(\overline{K_v})- \Pr(\overline{K_u} \cap \overline{K_v})$.
	We know that $$\Pr(\overline{K_u}) + \Pr(\overline{K_v}) =  (|N(u,c)|+|N(v,c)|-2)  \frac{\p}{L_i} \R_i,$$ so we just need to
	bound $\Pr(\overline{K_u} \cap \overline{K_v})$.
	
	$\Pr(\overline{K_u} \cap \overline{K_v})$ is the probability that there is some pair of non-incident
	edges $e_1 = uw \in N(u,c)\sm e$, and $e_2 = vx \in N(v,c)\sm e$ such that $e_1$ and $e_2$ are
	both activated, receive and retain $c$ during the procedure.
	Note that, as $e_1$ and $e_2$ are non-incident, this is $w \neq x$. 
	Now, for each such pair, we let $R_{e_1 ,e_2}$ be the event that $e_1$ and $e_2$ both retain $c$ after step~\ref{step:unassign-and-coinflip}.
	It follows that
	\begin{align*}
		\Pr(R_{e_1 ,e_2})= \left(\frac{\p}{L_i}\right)^2 \left(1- \frac{\p}{L_i}\right)^{|N(v,c) \cup N(w,c) \cup N(v,c)  \cup N(x,c)|-2} (1-\Eq(e_1,c)) (1-\Eq(e_2,c)).
	\end{align*}
	Since $|N(u,c) \cup N(w,c) \cup N(v,c)  \cup N(x,c)|< |N(u,c) | + |N(w,c)| -1 + |N(v,c)|  +| N(x,c)| -1$,
	we obtain:
	\begin{align*}
		\Pr(R_{e_1 ,e_2}) >   \left(\frac{\p}{L_i}  \R_i^2 \right)^2.
	\end{align*}
	
	It is easy to see that there are at most $N_i$ incident pairs $e_1 \in N(u,c) \sm e$, $e_2 \in N(v,c) \sm e$, as each edge in $N(u,c) \sm e$ is incident with at most one edge in $N(v,c) \sm e$.
	Therefore the number of non-incident pairs $e_1 \in N(u,c) \sm e$, $e_2 \in N(v,c) \sm e$ is at
	least  $(|N(u,c)|-1)(|N(v,c)|-1)-N_i$.
	
	For any two distinct pairs $( e_1, e_2)$ and $( e_1',e_2')$, it is impossible for $R_{e_1 ,e_2}$
	and $R_{e_1' ,e_2'}$ to both hold. 
	Therefore the probability that $R_{e_1 ,e_2}$ holds for at least one non-incident pair, is equal to the sum over all non-incident pairs $e_1, e_2$ of $\Pr(R_{e_1 ,e_2})$, which by the above remarks yields:
	\begin{align*}
		\Pr(\overline{K_u} \cap \overline{K_v}) &\geq  \left((|N(u,c)|-1)(|N(v,c)|-1)-N_i\right)  \left(\frac{\p}{L_i}  \R_i^2 \right)^2
	\end{align*}
	Combining this with our bound  on $\Pr(\overline{K_u})+ \Pr(\overline{K_v})$, we see that:
	\begin{align*}
		\Pr( {K_u} \cup {K_v}) &\geq 1 -   (|N(u,c)|+|N(v,c)|-2)  \frac{\p}{L_i} \R_i  
		\\&\quad+ \left((|N(u,c)|-1)(|N(v,c)|-1)-N_i\right)  \left(\frac{\p}{L_i}  \R_i^2 \right)^2
		\\&\geq \left(1-   \frac{ \p}{L_i} (|N(u,c)|-1) \R_i^2 \right) \left(1-   \frac{\p}{L_i}(|N(v,c)|-1) \R_i^2 \right) - \left(\frac{\p}{L_i}  \R_i^2 \right)^2
		\\&\geq \left(1-  \frac{ \p}{L_i} (|N(u,c)|-1) \R_i^2 \right) \left(1-  \frac{\p}{L_i}(|N(v,c)|-1)  \R_i^2 \right) . 
	\end{align*}
	Recall that definition of $\Vq$ in~\eqref{equ:Vq}.
	It follows  that
	\begin{align*}
		\Pr( c \in L'(e)) &=\Pr(K_u \cup K_v)(1-\Vq(u,c))  (1-\Vq(v,c))  
		\\&\geq \left(1 - \p \frac{|N(u,c)|}{L_i} \R_i^2\right)(1-\Vq(u,c)) \cdot  \left(1 - \p \frac{|N(v,c)|}{L_i} \R_i^2\right)(1-\Vq(v,c)) 
		\\&= \K_i^2,
	\end{align*}
	as desired.
\end{proof}

It remains to deal with $N'(v,c)$.
As it turns out, this random variable is not concentrated around its expected value.
The reason for this is that, if colour $c$ is assigned to one of the edges of $N(v,c)$ then the size of $|N'(v,c)|$ drops immediately to zero.
We will therefore carry out the expectation and concentration details for a different variable $N^*(v,c)$, which ignores the assignments to edges of $N(v,c)$. 
More precisely, let $N^*(v,c)$ be the set of edges $uv \in N(v,c)$ such that $uv$ does not retain a colour and no edge in $N(u,c)$ retains $c$.
Since $N'(v,c) \subset N^*(v,c)$, it suffices to focus the analysis on $N^*(v,c)$.

\begin{claim}[{\cite[Claim 14.10]{MR13}}]
	We have	$\Exp(|N^*(v,c)|)  \leq N_i \cdot  \K_i  \cdot \left(1- \p \R_i^2   \right) +1$ for every vertex $v$ and colour $c$.
\end{claim}
\begin{proof}
	We will show that for each $e = uv$ in $N(v,c)$, we have: $\Pr(e \in N^*(v,c)) \leq \left(1- \p \R_i^2   \right) + \frac{1}{L_i}$.
	As before this implies the result by linearity of expectation.
	
	We define $A$ to be the event that $e$ does no retain its colour and $B$ to be the event that no edge in $N(u,c)$ retains $c$.
	We wish to bound $\Pr(A \cap B)$.
	Once again, we proceed in an indirect manner, and focus instead on $\Pr(\overline{A} \cap \overline{B})$, showing that $\Pr(\overline{A} \cap \overline{B}) \leq \p^2 \frac{|N(u,c)|}{L_i} \R_i^4+ \frac{1}{L_i}$, thus implying
	\begin{align*}
		\Pr(A \cap B) &= \Pr(A) - \Pr(\overline{B}) + \Pr(\overline{A} \cap \overline{B}) 
		\\ &\leq \left(1-\p \R_i^2 \right) - \p \frac{|N(u,c)|}{L_i}\R_i^2 + \p^2 \frac{|N(u,c)|}{L_i} \R_i^4+ \frac{1}{L_i}
		\\ &= \left(1-\p \frac{|N(u,c)|}{L_i}\R_i^2 \right) \left(1-\p \R_i^2 \right) + \frac{1}{L_i}.
	\end{align*}
	Therefore 
	\begin{align*}
		\Pr(e \in N^*(v,c)) = \Pr(A \cap B) (1-\Vq(u,c)) \leq \K_i \left(1-\p \R_i^2 \right) + \frac{1}{L_i}.
	\end{align*}
	
	For each colour $d \in L(e)$ and edge $f=uw$ in $N(u,c) \sm e$, we define $Z(d,f)$ to be the event that $e$ retains $d$ and $f$ retains $c$.
	For each $d\neq c$, we have
	\begin{align*}
		\Pr(Z(d,f)) &= \left(\frac{\p}{L_i}\right)^2 \left(1-\frac{2\p}{L_i}\right)^{|\left[(N(v,d) \cap N(w,c))\cup (N(u,d)\cap N(u,c)) \right] \sm \{e,f\}|}
		\\&\quad \cdot \left(1-\frac{\p}{L_i}\right)^{|\left[N(v,d) \cup N(w,c)\cup N(u,d)\cup N(u,c) \right] \sm [(N(v,d) \cap N(w,c))\cup (N(u,d)\cap N(u,c)) \cup \{e,f\}]|}
		\\&\quad \cdot (1-\Eq(e,d)) (1-\Eq(f,c))
		\\&\leq \left(\frac{\p}{L_i}\right)^2
		\\&\quad \cdot \left(1-\frac{\p}{L_i}\right)^{|\left[N(v,d) \cup N(w,c)\cup N(u,d)\cup N(u,c) \right] \sm \{e,f\}| +|([N(v,d) \cap N(w,c))\cup (N(u,d)\cap N(u,c))] \sm \{e,f\}|}
		\\&\quad \cdot (1-\Eq(e,d)) (1-\Eq(f,c)).
	\end{align*}
	Note that
	\begin{align*}
		&\quad |\left[N(v,d) \cup N(w,c)\cup N(u,d)\cup N(u,c) \right] \sm \{e,f\}| 
		\\&\quad +|([N(v,d) \cap N(w,c))\cup (N(u,d)\cap N(u,c))] \sm \{e,f\}| 
		\\&= |N(v,d)\sm \{e\}| + |N(u,d)\sm \{e,f\}| + |N(u,c)\sm \{e,f\}| +  |N(w,c)\sm \{f\}|
		\\&\geq |N(v,d)| + |N(u,d) | + |N(u,c) | +  |N(w,c) | - 6.
	\end{align*}
	Hence, we obtain
	\begin{align*}
		\Pr(Z(d,f)) \leq \left(\frac{\p}{L_i} \right)^2 \left(1-\frac{\p}{L_i}\right)^{-2}  \R_i^4 = \frac{\p^2}{(L_i -\p)^2} \R_i^4.
	\end{align*}
	Since the events $Z(d,f)$ are disjoint, we have:
	\begin{align*}
		\Pr(\overline{A} \cap \overline{B}) &= \frac{\p}{L_i} \R_i^2 + \sum_{d \in L(e)\sm\{c\},f\in N(u,c)} \Pr(Z(d,f))
		\\&\leq \frac{\p}{L_i} \R_i^2 + L_i |N(u,c)| \frac{\p^2}{(L_i -\p)^2} \R_i^4  
		\\&< \frac{1}{L_i} + \p^2\frac{|N(u,c)|}{L_i}\R_i^4.
	\end{align*}
\end{proof}

\subsection{Concentration}\label{sec:MR-concentration}
The following three claims contain the desired concentration bounds.

\begin{claim}\label{cla:concentration-R}
	For every vertex $v$ and colour $c$, we have
	\begin{align*} 	
		\Pr(\left||R'(v,c)|-\Exp(|R'(v,c)|)\right|> \tfrac{1}{2} \sqrt{R_i} \log^2 \Delta) \leq \frac{\Delta^{-10\log \Delta}}{3}.
	\end{align*}
\end{claim}
\begin{proof}
	Fix a vertex $v$, a colour $c$ and let $R' = |R'(v,c)|$.
	We wish to apply Talagrand's inequality (Theorem~\ref{thm:talagrand}) to $R'$ with $c=2$, $r=1$ and $t=\tfrac{1}{2}\sqrt{R_i} \log^2\Delta$.
	
	Note that changing any the outcome of the edge activation in step~\ref{step:activate}, colour assignment in step~\ref{step:assign} or coin flip in step~\ref{step:unassign-and-coinflip} and~\ref{step:update-lists-and-coinflip} can affect the size of $R'$ by at most $2$.
	(Here it is important to recall that at most one edge of $R(v,c)$ can retain a specific colour at the same time.)
	Moreover, if $R'\geq s$, then for every edge $e \in R'(v,c)$ there is either an activation event or an assignment of a colour to an edge incident to $e$ that witnesses $e$ not retaining a colour.
	Recall that by Lemma~\ref{lem:sizes}, $R_i \geq \log^7 \Delta$.
	Moreover, $\Exp(R') \leq R_i \leq L_i$.
	Since $$t=\tfrac{1}{2}\sqrt{R_i} \log^2\Delta \geq 96c \sqrt{r \Exp(R')} + 128rc^2,$$ we obtain from Theorem~\ref{thm:talagrand} that 
	\begin{align*}
	\Pr(|R'-\Exp(R')| > t) \leq 4 \exp\left({-\frac{t^2}{8c^2r(4\Exp(R')+t)}}\right)\leq \exp\left(-\frac{t}{2^8R_i}\right) \leq \frac{\Delta^{-10\log \Delta}}{3},
	\end{align*}
	where we used that $t \leq 4\Exp(R') \leq 4R_i$ in the second inequality.
\end{proof}

\begin{claim}\label{cla:concentration-L}
	For every edge $e$, we have 
	\begin{align*} 	
	\Pr(\left||L'(e)|-\Exp(|L'(e)|)\right|> \tfrac{1}{2} \sqrt{L_i} \log^2 \Delta) \leq \frac{\Delta^{-10\log \Delta}}{3}.
	\end{align*}
\end{claim}
\begin{proof}
	Fix an edge $e=uv$ and let $L' = |L'(e)|$.
	Let $X$ be the number of colours $c \in L(e)$, which are retained by at least one edge in $(N(u,c) \cup N(v,c))\sm\{e\}$.
	For $0 \leq k \leq j \leq 2$, we define $Y_{j,k}$ to be the number of colours which are assigned to an edge in \emph{exactly} $j$ of $N(u,c)\sm \{e\}$, $N(v,c)\sm \{e\}$ and which are removed from an edge in at least $k$ of $N(u,c)\sm \{e\}$, $N(v,c)\sm \{e\}$ during step~\ref{step:assign} and~\ref{step:unassign-and-coinflip} of the procedure.
	Similarly, we define $X_{j,k}$ to be the number of colours which are assigned to an edge in \emph{at least} $j$ of $N(u,c)\sm \{e\}$, $N(v,c)\sm \{e\}$ and which are removed from an edge in at least $k$ of $N(u,c)\sm \{e\}$, $N(v,c)\sm \{e\}$ during step~\ref{step:assign} and~\ref{step:unassign-and-coinflip} of the procedure.
	Note that $Y_{2,k} = X_{2,k}$ and for $j<2$, $Y_{j,k} = X_{j,k} - X_{j+1,k}$.
	Making use of the very useful fact that for any $w$, if a colour is removed from at least one edge in $N(v,c)$, then it is removed from every edge in $N(v,c)$ to which it is assigned, we obtain that:
	\begin{align*}
		X=(Y_{2,0} - Y_{2,2}) + (Y_{1,0} - Y_{1,1}) = (X_{2,0} - X_{2,2}) + ((X_{1,0} - X_{2,2})-(X_{1,1} - X_{2,1})).
	\end{align*}
	
	Fix $1\leq j,k\leq 2$. We will show that  $X_{j,k}$ is highly concentrated.
	To this end, we  apply Talagrand's inequality (Theorem~\ref{thm:talagrand}) to $X_{j,k}$ with $c=2$, $r=4$ and $t=  \tfrac{1}{14}\sqrt{L_i} \log^2\Delta$.
	First of all, changing the colour assigned to any one edge from $c_1$ and $c_2$ can only affect whether $c_1$ and/or $c_2$ are counted by $X_{j,k}$, and changing the decision to uncolour an edge in step~\ref{step:unassign-and-coinflip} can only affect whether the colour of the edge is counted by $X_{j,k}$.
	Secondly, if $X_{j,k} \geq s$, then there is a set of at most $s(j+k)$ outcomes which certify this fact, namely for each of the $s$ colours, $j$ edges on which that colour appears, along with $k$ (or fewer) outcomes which cause $k$ of those edges to be uncoloured.
	Recall that by Lemma~\ref{lem:sizes}, $L_i \geq \log^7 \Delta$.
	Moreover, $\Exp(X_{j,k}) \leq  2N_i \leq 2L_i$.
	Since $$t=\tfrac{1}{14}\sqrt{L_i} \log^2\Delta \geq 96c \sqrt{r \Exp(X_{j,k})} + 128rc^2,$$ we obtain from Theorem~\ref{thm:talagrand} that 
	\begin{align*}
	\Pr(|X_{j,k}-\Exp(X_{j,k})| > t) \leq 4 \exp\left({-\frac{t^2}{8c^2r(4\Exp(X_{j,k})+t)}}\right)\leq \exp\left(-\frac{t^2}{2^{11} L_i}\right) \leq \frac{\Delta^{-10\log \Delta}}{21},
	\end{align*}
	where we used that $t \leq 4\Exp(X_{j,k}) \leq 8L_i$ in the second inequality.
	
	Finally, let $X'$ be the number of colours removed from $\bigcup_{f \in N(u,c)} L(f)$ or  $\bigcup_{f \in N(v,c)} L(f)$ in step~\ref{step:update-lists-and-coinflip}.
	It follows easily from Talagrand's inequality (or, simpler, Chernoff's inequality) that $X'$ is  highly concentrated around its expectation.
	(Note that our constant $t$ and the bound on the probability allows for the sevenfold error term.)
	Since $|L'(e)| = |L(e)| - (X + X')$, this concludes the proof.
\end{proof}

\begin{claim}
	For every vertex $v$ and colour $c$, we have
	\begin{align*} 	
	\Pr(\left||N^*(v,c)|-\Exp(|N^*(v,c)|)\right|> \tfrac{1}{2} \sqrt{N_i} \log^2 \Delta) \leq \frac{\Delta^{-10\log \Delta}}{3}.
	\end{align*}
\end{claim}
\begin{proof}
	We again proceed indirectly.
	We let $A_{v,c}$ be the set of edges $e \in N(v,c)$, which do not retain their colours.
	We let $B_{v,c}$ be the set of edges $e = uv$ in $A_{v,c}$ such that $c$ is retained on some vertex of $N(u,c)$.
	We let $C_{v,c}$ be the set of edges $e =uv$ in $A_{v,c} \sm B_{v,c}$ such that $c$ is removed from  the lists $L(f)$ of all edges $f \in N(v,c)$ because of an equalizing coin flip.
	
	The proof that $|A_{v,c}|$ is highly concentrated is virtually identical to the proof Claim~\ref{cla:concentration-R}.
	The proof that $|B_{v,c}|$ and $|B_{v,c}|$ are highly concentrated follows along the lines of the proof of Claim~\ref{cla:concentration-L}.
	Since $T'(u,c) = A_{v,c} \sm (B_{v,c} \cup C_{v,c})$, the desired result follows.
\end{proof}

\end{document}